\documentclass[11pt,letterpaper]{amsart}

\usepackage[english]{babel}
\usepackage[dvipsnames]{xcolor}
\usepackage{xstring}
\usepackage{fontawesome5}

\usepackage[backgroundcolor=white,linecolor=red,bordercolor=red]{todonotes}
\usepackage[shortlabels]{enumitem}
\setenumerate{label=(\roman*)}

\usepackage{hyperref}
\newcommand\myshade{100}
\hypersetup{
  linkcolor  = RoyalBlue!\myshade!black,
  citecolor  = ForestGreen!\myshade!black,
  urlcolor   = RedOrange!\myshade!black,
  colorlinks = true,
}

\colorlet{githubrefcolor}{RedOrange!\myshade!black}
\makeatletter
\newcommand{\github@breakable@text}[1]{%
  \begingroup
  \StrSubstitute{#1}{/}{\discretionary{/}{}{/}}[\github@breakable@result]%
  \texttt{\github@breakable@result}%
  \endgroup
}
\newcommand{\github@breakable@literal}[1]{%
  \begingroup
  \edef\github@text{\detokenize{#1}}%
  \github@breakable@text{\github@text}%
  \endgroup
}
\newcommand{\github@ref}[3]{%
  \href{#3}{\textcolor{githubrefcolor}{#1\,#2}}%
}
\newcommand{\github@ref@literal}[3]{%
  \href{#3}{\textcolor{githubrefcolor}{#1\,\github@breakable@literal{#2}}}%
}
\newcommand{\github@faicon}[1]{{\scriptsize #1}}
\DeclareRobustCommand{\leanicon}{%
  \begingroup
  \scriptsize
  \leavevmode
  \ooalign{%
    \hfil$\bm\forall$\hfil\cr
    \hfil\kern0.018em$\bm\forall$\hfil\cr
    \hfil\kern0.036em$\bm\forall$\hfil\cr
    \hfil\kern-0.018em$\bm\forall$\hfil\cr
    \hfil\kern-0.036em$\bm\forall$\hfil\cr
  }%
  \endgroup
}
\newcommand{\github@strip@suffix}[1]{%
  \IfSubStr{#1}{?}{\StrBefore{#1}{?}[#1]}{}%
  \IfSubStr{#1}{\#}{\StrBefore{#1}{\#}[#1]}{}%
}
\newcommand{\github@repo@text}[1]{%
  \begingroup
  \StrBehind{#1}{github.com/}[\github@path]%
  \github@strip@suffix{\github@path}%
  \IfSubStr{\github@path}{/}{%
    \StrBehind{\github@path}{/}[\github@afterowner]%
    \IfSubStr{\github@afterowner}{/}%
      {\StrBefore{\github@afterowner}{/}[\github@text]}%
      {\let\github@text\github@afterowner}%
  }{%
    \let\github@text\github@path
  }%
  \github@breakable@text{\github@text}%
  \endgroup
}
\newcommand{\github@line@text}[1]{%
  \begingroup
  \StrBehind{#1}{/blob/}[\github@path]%
  \github@strip@suffix{\github@path}%
  \IfSubStr{\github@path}{/}%
    {\StrBehind{\github@path}{/}[\github@file]}%
    {\StrBehind{#1}{github.com/}[\github@file]\github@strip@suffix{\github@file}}%
  \StrBehind{#1}{\#}[\github@line]%
  \IfSubStr{\github@line}{?}{\StrBefore{\github@line}{?}[\github@line]}{}%
  \github@breakable@text{\github@file}\space\github@breakable@text{\github@line}%
  \endgroup
}
\newcommand{\github@repo@default}[1]{%
  \github@ref{\github@faicon{\faGithub}}{\github@repo@text{#1}}{#1}%
}
\newcommand{\github@line@default}[1]{%
  \github@ref{\leanicon}{\github@line@text{#1}}{#1}%
}
\newcommand{\github@auto@default}[1]{%
  \IfSubStr{#1}{\#L}%
    {\github@line@default{#1}}%
    {\github@repo@default{#1}}%
}
\newcommand{\github@auto@literal}[2]{%
  \IfSubStr{#2}{\#L}%
    {\github@ref@literal{\leanicon}{#1}{#2}}%
    {\github@ref@literal{\github@faicon{\faGithub}}{#1}{#2}}%
}
\NewDocumentCommand{\gh}{om}{%
  \IfNoValueTF{#1}%
    {\github@auto@default{#2}}%
    {\github@auto@literal{#1}{#2}}%
}
\makeatother

\usepackage{bm}

\usepackage[capitalise]{cleveref}
\crefformat{equation}{(#2#1#3)}
\crefname{subsection}{Section}{Sections}

\theoremstyle{plain}
\newtheorem{thm}{Theorem}[section]
\crefname{thm}{Theorem}{Theorems}
\newtheorem{lem}[thm]{Lemma}
\newtheorem{prop}[thm]{Proposition}

\theoremstyle{definition}
\newtheorem{defn}[thm]{Definition}
\newtheorem{question}{Question}

\ExplSyntaxOn
\NewDocumentCommand{\DeclareLeanTheoremEnvironment}{m}
  {
    \cs_new_eq:cc { lean_old_#1: } { #1 }
    \cs_new_eq:cc { lean_old_end_#1: } { end#1 }
    \RenewDocumentEnvironment{#1}{ o d<> }
      {
        \IfNoValueTF{##1}
          { \use:c { lean_old_#1: } }
          { \use:c { lean_old_#1: } [##1] }
        \IfNoValueF{##2}
          { \nobreak\textup{##2}\space }
      }
      { \use:c { lean_old_end_#1: } }
  }
\ExplSyntaxOff

\DeclareLeanTheoremEnvironment{thm}

\makeatletter
\def\@fnsymbol#1{\ensuremath{\ifcase#1\or \dagger\or \ddagger\or
           \dagger\dagger
           \or \ddagger\ddagger \else\@ctrerr\fi}}
\makeatother

\newcommand\N{\ensuremath{\mathbb{N}}}
\newcommand\R{\ensuremath{\mathbb{R}}}
\newcommand\Q{\ensuremath{\mathbb{Q}}}

\DeclareMathOperator{\supp}{supp}

\usepackage[backend=biber,maxnames=999,giveninits=true,style=numeric,sortcites=true,datamodel=mrnumber,isbn=false,url=false,doi=false]{biblatex}
\usepackage{csquotes}

\renewbibmacro*{in:}{}

\DeclareRobustCommand{\theoremcite}[2]{\cite[#1]{#2}}

\DeclareFieldFormat{title}{\myhref{\mkbibemph{#1}}}
\DeclareFieldFormat
  [article,inbook,incollection,inproceedings,patent,thesis,unpublished]
  {title}{\myhref{\mkbibquote{#1\isdot}}}

\newcommand{\doiorurl}{%
  \iffieldundef{doi}
    {\iffieldundef{url}
       {}
       {\strfield{url}}}
    {https://doi.org/\strfield{doi}}%
}

\newcommand{\myhref}[1]{%
 \ifboolexpr{%
   test {\ifhyperref}
   and
   not test {\iftoggle{bbx:url}}
   and
   not test {\iftoggle{bbx:doi}}
  }
  {\ifboolexpr{%
     test {\iffieldundef{doi}}
     and
     test {\iffieldundef{url}}
   }
   {#1}
   {\href{\doiorurl}{#1}}}
  {#1}%
}

\DeclareFieldFormat{mrnumber}{%
  MR\addcolon\space
  \ifhyperref
    {\href{https://www.ams.org/mathscinet-getitem?mr=#1}{\nolinkurl{#1}}}
    {\nolinkurl{#1}}}

\renewbibmacro*{doi+eprint+url}{%
  \printfield{mrnumber}%
}

\DeclareFieldFormat{eid}{\printtext{article no.} #1}

\title[Aliprantis's questions]{Aliprantis's questions on locally solid topologies}
\author{David Muñoz-Lahoz}
\address{Instituto de Ciencias Matemáticas, Universidad Autónoma de
Madrid, C/ Nicolás Cabrera 13–15, 28049 Madrid, Spain }
\email{david.munnozl@uam.es}
\author{Mitchell A.\ Taylor}
\address{Department of Mathematics, ETH Zürich, Rämistrasse 101, 8092 Zürich, Switzerland}
\email{mitchell.taylor@math.ethz.ch}
\author{Pedro Tradacete}
\address{Instituto de Ciencias Matemáticas, Consejo Superior de
    Investigaciones Científicas, C/ Nicolás Cabrera 13–15, 28049 Madrid, Spain }
\email{pedro.tradacete@icmat.es}
\thanks{First author supported by an FPI–UAM 2023 contract (funded by
Universidad Autónoma de Madrid). First and third authors were partially supported by grants PID2024-162214NB-I00 and
CEX2023-001347-S (funded by MCIN/AEI/10.13039/501100011033).}
\date{\today}

\begin{document}

\begin{abstract}
    In 1974, C.\ D.\ Aliprantis posed several questions concerning
    topological completions of locally solid vector lattices. We
    answer all of them in the negative. We construct Hausdorff
    locally convex-solid vector lattices showing that, without
    metrizability, neither the $\sigma$-Lebesgue property nor property
    (B, i) need pass to the completion; a positive element of the
    completion need not be the limit of a decreasing sequence of upper
    elements; the generalized (A, 0) property need not make the
    canonical image a regular sublattice; and regularity of this image
    need not imply order density. All five
    counterexamples have been formalized in Lean 4 using
    the Banach lattice Lean library.
\end{abstract}

\maketitle
\section{Introduction}

C.\ D.\ Aliprantis's early work focused on locally solid vector lattices
\cite{wickstead2011}. This line of research culminated in the
publication of the standard monograph in the field, coauthored with
O.\ Burkinshaw \cite{aliprantis_burkinshaw2003}. In one of his earliest
papers, published in 1974, Aliprantis studied completions of locally
solid vector lattices \cite{aliprantis1974}. At the end of that paper,
he posed a list of open questions whose resolution would help complete
the theory. Many of these questions were also stated in the original 1978 edition of \cite{aliprantis_burkinshaw2003}
and in the memorial paper \cite{wickstead2011}. 
The goal of this note is to answer all of them in
the negative by constructing counterexamples and to provide a Lean
formalization of the counterexamples. For each question, we shall also
provide enough context to motivate it and show how it fits into the
theory.

Next, we give an overview of the questions considered. The questions
in \cref{sec:q1,sec:q2,sec:q3,sec:q5} concern properties or
implications known to hold in the metrizable case, and Aliprantis
asked whether they also hold in the non-metrizable case. More
precisely, \cref{sec:q1,sec:q2} show that the $\sigma$-Lebesgue
and (B, i) properties (see Definitions \ref{def:sigmaLebesgue} and \ref{def:property(B,i)}), respectively, need not pass to the completion in the
non-metrizable case. On the other hand, \cref{sec:q3,sec:q5} show that the metrizability assumptions in
\cref{thm:lux,thm:regdense}, respectively, cannot be dropped. Finally,
the question addressed in \cref{sec:q4} is slightly different: it is shown that
the condition in \cref{thm:kawai}, which characterizes when the image
of a Hausdorff locally solid vector lattice in its completion is
regular, cannot be in general weakened.

Aliprantis, together with O.\ Burkinshaw, posed two further questions on locally solid vector
lattices in a later paper \cite{aliprantis_burkinshaw1977}. Although
they are still listed as open questions in \cite{wickstead2011}, their
consistency was actually clarified by D.\ H.\ Fremlin \cite{fremlin1975},
G.\ Buskes, and I.\ Labuda \cite{buskes_labuda1988}, who showed that
an affirmative answer to either question is equiconsistent with the
nonexistence of measurable cardinals (cf. \cite[Chapter 10]{Jech}).

\subsection{Preliminaries and notation}

For terminology concerning vector lattices (also called Riesz spaces)
and locally solid topologies that is not introduced here, we refer the
reader to the monograph \cite{aliprantis_burkinshaw2003}. In the
remainder of this section, we recall some standard facts and notation
that will be used throughout.

Given a vector space $X$ and a family $\mathcal{P}$ of seminorms on
$X$, the topology generated by $\mathcal{P}$ is the unique locally
convex linear topology for which the sets of the form
\[
\{\, x \in X : p_1(x),\ldots ,p_n(x)<\varepsilon
\, \},
\]
where $n \in \N$, $p_1,\ldots ,p_n \in \mathcal{P}$ and $\varepsilon
>0$, form a basis of neighborhoods of zero. This topology is Hausdorff
if and only if, for every nonzero $x \in X$, there exists
$p \in \mathcal{P}$ such that $p(x)\neq 0$. If $X$ is a vector lattice
and every element of $\mathcal{P}$ is a lattice seminorm, then $X$,
equipped with the topology generated by $\mathcal{P}$, is a locally
convex-solid vector lattice.

Throughout, a Cauchy net is a topologically Cauchy net; that is, a net
$(x_\alpha )\subseteq X$ such that, for every neighborhood of zero
$U$, there exists an index $\alpha _0$ for which
$x_\alpha -x_\beta \in U$ for all $\alpha ,\beta \ge \alpha _0$.
Similarly, a convergent net is always understood to be topologically
convergent, and $x_\alpha \to x$ always denotes topological
convergence. A topological vector space is topologically complete if
every Cauchy net converges.

Let $X$ be a Hausdorff topological vector space. Suppose that there is
a complete Hausdorff topological vector space $\hat{X}$ and an
embedding $J\colon X\to \hat{X}$ (that is, a linear topological
embedding) such that $J(X)$ is dense in $\hat{X}$. Then the pair
$(\hat{X},J)$ is unique up to isomorphism and is called the
topological completion of $X$. Although $X$ can be identified with
$J(X)$, and hence regarded as a linear subspace of $\hat{X}$, we will
usually display $J$ explicitly to avoid ambiguity.

Every Hausdorff topological vector space admits a topological
completion \cite[Section 1.5]{schaefer1966}. If $X$ is a Hausdorff
locally solid vector lattice, then the closure of $J(X_+)$ in
$\hat{X}$ equips $\hat{X}$ with the structure of a Hausdorff locally
solid vector lattice and makes $J$ a lattice homomorphism. Henceforth,
when we refer to the completion of a Hausdorff locally solid vector
lattice, we mean this particular Hausdorff locally solid vector
lattice structure.

In practice, to verify that a Hausdorff locally solid vector lattice
$Y$, together with a map $J\colon X\to Y$, is the completion of $X$,
it suffices to check the following:
\begin{enumerate}
    \item $Y$ is complete.
    \item $J$ is a vector lattice homomorphism and a topological
        embedding.
    \item $J(X_+)$ is dense in $Y_+$.
\end{enumerate}
The third condition ensures that $Y_+$ is the closure of $J(X_+)$ and
that $J(X)$ is dense in $Y$. We will use this checklist repeatedly throughout the paper.

\subsection{Lean formalization}

The counterexamples presented in this paper have been formalized in
Lean 4 using version \verb|v0.1.0| of
\gh{https://github.com/davidmunozlahoz/banlat}, a Lean library for
Banach lattices. This version of the library, in turn, depends on Mathlib version
\verb|v4.30.0| \gh{https://github.com/leanprover-community/mathlib4}.
Each of the five sections addressing Aliprantis's questions
corresponds to a file in the formalization repository
\gh{https://github.com/davidmunozlahoz/aliprantis}.
The main theorem in each section includes a link to the corresponding
Lean declaration. These links point to a specific line in
\gh{https://github.com/davidmunozlahoz/aliprantis} and are displayed
as follows:
\gh[theorem2_5]{https://github.com/davidmunozlahoz/Aliprantis/blob/526881fef80400db04f8f3c736dc64d3fdc00f6c/Aliprantis/Q1.lean\#L1008}.

\section{Question 1}\label{sec:q1}

\subsection{Background}

We begin by recalling two basic definitions in the theory of locally
solid vector lattices.

\begin{defn}\label{def:sigmaLebesgue}
    Let $X$ be a locally solid vector lattice.
    \begin{enumerate}
        \item We say that $X$ satisfies the \emph{$\sigma$-Lebesgue
            property} if, for every sequence $(x_n)\subseteq X_+$,
            $x_n \downarrow 0$ implies $x_n \to 0$.
        \item We say that $X$ satisfies the \emph{Lebesgue
            property} if, for every net $(x_\alpha )\subseteq X_+$,
            $x_\alpha \downarrow 0$ implies $x_\alpha \to 0$.
    \end{enumerate}
\end{defn}

These properties were called (A, i) and (A, ii), respectively, in
\cite{luxemburg_zaanen1964_X} and in the reference paper
\cite{aliprantis1974}. We use instead the terminology of
\cite{aliprantis_burkinshaw2003}, which originates in a paper by
D.\ H.\ Fremlin \cite{fremlin1974}; see the historical note following
\cite[Definition 3.1]{aliprantis_burkinshaw2003}. These two properties
have been studied extensively (see \cite[Chapter 3]{aliprantis_burkinshaw2003}),
especially when $X$ is a Banach lattice equipped with its norm
topology---that is, when $X$ is a
($\sigma$-)order continuous Banach lattice \cite{wnuk1999}. In
\cite{aliprantis1974}, Aliprantis investigated whether the
($\sigma$-)Lebesgue property passes to the topological completion and
obtained the following results.

\begin{thm}[\theoremcite{Theorem 3.2}{aliprantis1974} and
    \theoremcite{Theorem 5.1}{aliprantis1974}]
    Let $X$ be a Hausdorff locally solid vector lattice.
    \begin{enumerate}
        \item If $X$ is Lebesgue, then its topological
            completion is also Lebesgue.
        \item If $X$ is metrizable and $\sigma$-Lebesgue, then its
            topological completion is also $\sigma$-Lebesgue.
    \end{enumerate}
\end{thm}

These results motivate the following natural question.

\begin{question}[\theoremcite{Open Problem 1}{aliprantis1974}]
    Is the $\sigma$-Lebesgue property preserved under topological
    completion without assuming metrizability?
\end{question}

\subsection{Solution}

We construct a Hausdorff locally solid vector lattice that is
$\sigma$-Lebesgue but whose topological completion is not
$\sigma$-Lebesgue.

Let $A$ be an uncountable set with the discrete topology, and let
\[
A_\infty =A\cup \{\infty \}
\]
be its one-point compactification.
Recall that a function $f\colon A_\infty \to \R$ is continuous if and
only if, for every $\varepsilon >0$, the set
\[
\{\, a \in A : |f(a)-f(\infty )|\ge \varepsilon  \, \}
\]
is finite. We consider the space $C(A_\infty )$ of real-valued continuous functions on $A_\infty$, equipped with the topology whose convergence
is uniform on a fixed countable subset of $A $ and pointwise
everywhere else. More
precisely, fix a countable infinite subset
$B=\{b_j \colon j \in \N\}\subseteq A$, denote $A'=A\setminus B$, and define
\[
q(f)=\sup_{j} |f(b_j)|\quad\text{and}\quad q_a(f)=|f(a)|
\]
for every $a \in A'$ and $f \in C(A_\infty )$. The functions $q$ and
$q_a$, for $a \in A'$, are lattice seminorms on $C(A_\infty )$. Let
$\tau$ be the locally convex-solid topology on $C(A_\infty )$
generated by $\{q\}\cup \{q_a\colon a \in A'\}$. This
topology is Hausdorff: if $f \in C(A_\infty )$ is nonzero at some
$a \in A'$, then $q_a(f)\neq 0$; if it is nonzero at some $b_j \in B$,
then $q(f)\neq 0$; and if $f(\infty )\neq 0$, then also $q(f)\neq 0$
because $f(b_j)\to f(\infty )$.

\begin{lem}
    The Hausdorff locally solid vector lattice $(C(A_\infty ),\tau )$
    has the $\sigma $-Lebesgue property.
\end{lem}
\begin{proof}
    Let $(f_n)\subseteq C(A_\infty )$ be a decreasing sequence
    satisfying $f_n \downarrow 0$. We show that $f_n\to 0$ in the
    $\tau$ topology; equivalently, that $q(f_n)\to 0$ and
    $q_a(f_n)\to 0$ for every $a \in A'$.

    First we show that $f_n(a)\to 0$ for every $a \in A$. Suppose not.
    Then there would exist $a \in A$ and $\varepsilon >0$ such that
    $f_n(a)\ge \varepsilon $ for infinitely many $n \in \N$; since
    $(f_n(a))$ is a decreasing sequence, actually $f_n(a)\ge
    \varepsilon $ for all $n \in \N$. But then $\varepsilon \chi
    _{\{a\}}$ is a nonzero lower bound for $f_n$, contradicting the
    assumption. This shows, in particular, that $q_a(f_n)=|f_n(a)|\to
    0$ for every $a \in A'$.

    We also have $f_n(\infty )\to 0$. Again, suppose this were not the
    case. Then there would exist an $\varepsilon >0$ such that
    $f_n(\infty )\ge \varepsilon $ for all $n \in \N$. By continuity
    at $\infty $, for each $n \in \N$ there exists a finite
    $F_n\subseteq A$ such that
    \[
    f_n(a)\ge \varepsilon /2\quad\text{for all }a \in A\setminus F_n.
    \]
    Since $A$ is uncountable, we may choose $a \in A\setminus
    \bigcup_{n \in \N} F_n$. Then $(\varepsilon /2)\chi _{\{a\}}$ is a
    nonzero lower bound of $(f_n)$, again contradicting the
    assumption.

    Thus we have a decreasing sequence $(f_n)$ of
    continuous functions that converges pointwise to $0$ on the
    compact space $A_\infty $. By Dini's theorem, $(f_n)$
    converges uniformly to $0$. In particular,  $q(f_n)=\sup_j
    |f_n(b_j)|$ tends to $0$.
\end{proof}

Let $c$ denote the space of real convergent sequences, with its usual
Banach lattice structure. We next show that $c \times
\R^{A'}$, with its natural structure as the product of two (complete) Hausdorff
locally solid vector lattices, is the completion of $(C(A_\infty ),
\tau )$.

\begin{lem}
    The Hausdorff locally solid vector lattice $c \times \R^{A'}$,
    together with the embedding
    \[
    \begin{array}{cccc}
    J\colon& C(A_\infty ) & \longrightarrow & c \times \R^{A'} \\
            & f & \longmapsto & ((f(b_j))_{j \in \N}, (f(a))_{a \in
            A'}) \\
    \end{array},
    \]
    is the topological completion of $(C(A_\infty ),\tau )$.
\end{lem}
\begin{proof}
    Recall that convergence in $c \times \R^{A'}$ is uniform on $c$
    and pointwise on $\R^{A'}$. It is straightforward to check this space is topologically complete.
    The map $J$ is a vector lattice homomorphism, and it is injective
    because $f(b_j)\to f(\infty )$. Moreover, by the definition of
    $\tau$, $J$ is a homeomorphism onto its image.

    It only remains to show that $J(C(A_\infty )_+)$ is dense in
    $(c\times \R^{A'})_+$. Let $g=(g_1,g_2) \in (c \times
    \R^{A'})_+$, and let $U$ be a neighborhood of $g$. There exist
    $\varepsilon >0$, $m \in \N$, and $a_1,\ldots ,a_m \in A'$ such that
    \begin{multline*}
    \{\, h=(h_1,h_2) \in c\times \R^{A'} :\\ \|g_1-h_1\|_\infty
    <\varepsilon,\ |g_2(a_i)-h_2(a_i)|<\varepsilon
    \text{ for }i=1,\ldots,m\,\}
    \end{multline*}
    is contained in $U$.
    Define $f\colon A_\infty \to \R$ by
    \begin{align*}
        &f(b_j)=g_1(j)\quad \text{for }j \in \N,\\
        &f(a_i)=g_2(a_i)\quad \text{for }i=1,\ldots ,m,\\
        &f(a)=\lim_j g_1(j)\quad \text{for }a \in A'\setminus
    \{a_1,\ldots ,a_m\},\\
        &f(\infty )=\lim_j g_1(j).
    \end{align*}
    Then $f \in C(A_\infty )_+$. Moreover, the first coordinate of
    $J(f)$ agrees with $g_1$, while the second agrees with $g_2$ on
    $\{a_1,\ldots ,a_m\}$. Hence
    $J(f) \in U$.
\end{proof}

However, $c\times \R^{A'}$ does not have the $\sigma$-Lebesgue
property. Indeed, for $n \in \N$, let $x_n \in c$ be the sequence
whose first $n$ entries are $0$ and whose remaining entries are $1$.
Then $x_n \downarrow 0$, but $\|x_n\|_\infty =1$ for every
$n \in \N$. Consequently, $(x_n,0)\downarrow 0$ in
$c\times \R^{A'}$, but this sequence does not converge to $0$ in the
product topology. In summary, we have proved the following.

\begin{thm}<\gh[theorem2_5]{https://github.com/davidmunozlahoz/Aliprantis/blob/526881fef80400db04f8f3c736dc64d3fdc00f6c/Aliprantis/Q1.lean\#L1008}>There exists a Hausdorff $\sigma$-Lebesgue locally convex-solid
    vector lattice whose topological completion is not
    $\sigma$-Lebesgue.
\end{thm}

\section{Question 2}\label{sec:q2}

\subsection{Background}

Aliprantis introduced the following property in
\cite[Section 1]{aliprantis1974} as a generalization of the analogous
property for normed vector spaces introduced by W.\ A.\ J.\ Luxemburg
and A.\ C.\ Zaanen in \cite[Section 34]{luxemburg_zaanen1964_X}.

\begin{defn}\label{def:property(B,i)}
    We say that a locally solid vector lattice $X$ satisfies
    \emph{property (B, i)} if every sequence $(x_n)\subseteq X_+$ that
    is topologically bounded and increasing is Cauchy.
\end{defn}

When $X$ is a Banach lattice equipped with its norm topology, property
(B, i) is equivalent to $X$ being a KB-space: every increasing,
norm-bounded sequence in $X_+$ converges. Property (B, i) was
introduced together with its analogue for nets, property (B, ii). They
are the topologically bounded counterparts of properties (A, iii) and
(A, iv), respectively, which require every order-bounded increasing
sequence and net to be Cauchy. The latter properties
were first introduced for normed vector lattices
\cite[Section 33]{luxemburg_zaanen1964_X} to prove results related to
a theorem of H.\ Nakano. Aliprantis then showed that property (B, i)
is preserved under topological completion for metrizable locally solid
vector lattices.

\begin{thm}[\theoremcite{Theorem 5.3}{aliprantis1974}]
    A metrizable locally solid vector lattice has property (B, i) if
    and only if its topological completion has property (B, i).
\end{thm}

Clearly, if the completion satisfies property (B, i), then so does the
original space. This leaves the converse implication open in the
non-metrizable case.

\begin{question}[\theoremcite{Open Problem 1}{aliprantis1974}]
    Is property (B, i) preserved under topological completion without
    assuming metrizability?
\end{question}

\subsection{Solution}

We construct a Hausdorff locally solid vector lattice with property
(B, i) whose completion does not have property (B, i). It will be a
vector sublattice of $c_0\times\R^{PS}$, where
$PS=(\R_+)^{\N}$ is the space of nonnegative sequences, and $c_0$ is
the Banach lattice of real sequences converging to zero equipped with
the uniform norm. With the product topology, $c_0\times\R^{PS}$ is a
complete Hausdorff locally convex-solid vector lattice.

For $j \in \N$, define $a_j=((a_j)_1,(a_j)_2)\in c_0\times\R^{PS}$ by
$(a_j)_1(m)=\delta _{jm}$ for $m\in \N$, where $\delta _{jm}$ denotes the Kronecker
delta, and $(a_j)_2(\alpha)=\alpha(j)$ for $\alpha\in PS$. Let $Y$ be the vector
sublattice generated by $\{\, a_j : j \in \N \, \}$. Let
\[
I=\{\, (0,f)\in c_0\times \R^{PS} : f\text{ is finitely supported} \, \}
\]
be the ideal of finitely supported functions on $PS$, extended by zero
on \(\N\). This is an ideal of $c_0\times\R^{PS}$, so $X=Y+I$ is a
vector sublattice of $c_0\times\R^{PS}$.
Indeed, if
$y\in Y$ and $v\in I$, then the vector-lattice inequality
\[
    \bigl|\,|y+v|-|y|\,\bigr|\leq |v|
\]
shows that $|y+v|-|y|\in I$. Consequently,
$|y+v|\in Y+I$. Since a vector subspace closed under taking moduli is a
vector sublattice, $X=Y+I$ is a vector sublattice of $c_0\times\R^{PS}$.
Equipped with the subspace
topology, $X$ is a Hausdorff locally convex-solid vector lattice.

\begin{lem}
    The topological completion of $X$ is $c_0\times\R^{PS}$, together
    with the natural inclusion.
\end{lem}
\begin{proof}
    It only remains to show that $X_+$ is dense in $(c_0\times\R^{PS})_+$.
    Let $z=(z_1,z_2)\in(c_0\times\R^{PS})_+$, and let $U$ be an open
    neighborhood of $z$. Then
    there exist $\varepsilon >0$, $m \in \N$, and $\alpha _1,\ldots
    ,\alpha _m \in PS$ such that
    \begin{multline*}
    \{\, w=(w_1,w_2)\in c_0\times\R^{PS} :\\
    \|z_1-w_1\|_\infty<\varepsilon,
    \ |z_2(\alpha _i)-w_2(\alpha _i)|<\varepsilon
    \text{ for }i=1,\ldots,m \, \}
    \end{multline*}
    is contained in $U$.
    By retaining only finitely many coordinates of $z_1$, choose a
    positive, finitely supported sequence $y'$ such that
    $\|y'-z_1\|_\infty<\varepsilon$. Let
    $y=\sum_{j \in \N} y'(j)a_j\in Y$, and define $(0,f)\in I$ by
    \[
    f(\alpha _i)=z_2(\alpha _i)-\sum_{j \in \N} y'(j)\alpha _i(j)
    \]
    for $i=1,\ldots,m$, and $f(\alpha)=0$ otherwise. Then
    $x=y+(0,f)\in X_+$. If $x=(x_1,x_2)$, then
    $\|x_1-z_1\|_\infty=\|y'-z_1\|_\infty<\varepsilon$ and
    $x_2(\alpha _i)=z_2(\alpha _i)$ for $i=1,\ldots,m$. Thus, $x\in U$.
\end{proof}

Since $c_0$ is not a KB-space, the space $c_0\times\mathbb{R}^{PS}$ does not have property (B, i):

\begin{lem}
    $c_0\times\R^{PS}$ does not have property (B, i).
\end{lem}
\begin{proof}
    For $n\in\N$, let $s_n\in c_0$ be the sequence whose first $n$
    entries are $1$ and whose remaining entries are $0$. Then the
    sequence $(s_n,0)$ in $c_0\times\R^{PS}$ is positive, increasing, and
    topologically bounded, but it is not Cauchy.
\end{proof}

The difficult part is to show that $X$ has property (B, i).

\begin{lem}
    $X$ has property (B, i).
\end{lem}
\begin{proof}
    Let $(x_n)\subseteq X_+$ be an increasing, topologically bounded
    sequence. We need to show that it is Cauchy. Write
    \[
        x_n=y_n+(0,z_n)=((x_n)_1,(x_n)_2),
    \]
    where $y_n=((y_n)_1,(y_n)_2)\in Y$ and $(0,z_n)\in I$. Thus
    $(x_n)_1=(y_n)_1$ and $(x_n)_2=(y_n)_2+z_n$. Notice that
    $((x_n)_2)$ is immediately Cauchy in the product topology of
    $\R^{PS}$: for every $\alpha\in PS$, the sequence
    $((x_n)_2(\alpha))$ is increasing and bounded, hence Cauchy. Thus,
    it only remains to show that $((x_n)_1)$ is Cauchy in the uniform
    norm.

    Since $((x_n)_1)$ is uniformly bounded, the sequence
    \[
        u(k)=\sup_n (x_n)_1(k)\qquad(k\in\N)
    \]
    belongs to $\ell_\infty$. If $((x_n)_1)$ is Cauchy, then $u$ is
    the limit of this sequence in the uniform norm, and therefore
    belongs to $c_0$. Conversely, if $u\in c_0$, then it follows from
    Dini's theorem that $((x_n)_1)$ converges uniformly to $u$ (when
    viewed as continuous functions on the one-point compactification of
    $\N$), and
    therefore that it is Cauchy. In summary, $((x_n)_1)$, and in fact
    $(x_n)$, is Cauchy if and only if $u\in c_0$.

    Suppose, to reach a contradiction, that $(x_n)$ is not Cauchy.
    Then $u\notin c_0$. Let
    \[
        E=\bigcup_{n\in\N}\supp z_n\subseteq PS.
    \]
    This set is countable, and for every $\alpha\in PS\setminus E$,
    we have $(x_n)_2(\alpha)=(y_n)_2(\alpha)$ for all $n\in\N$.
    We are going to use the fact that $u\notin c_0$ to construct a point
    $\alpha\in PS\setminus E$ and a subsequence $(x_{n_j})$ such that
    \[
        (x_{n_j})_2(\alpha)=(y_{n_j})_2(\alpha)>j
        \qquad(j\in\N),
    \]
    thus contradicting the topological boundedness of $(x_n)$.
    Since $u\notin c_0$, there is $\varepsilon>0$ such that, for every
    finite $F\subseteq\N$, some $k\in\N\setminus F$ satisfies
    $u(k)>\varepsilon$. Consequently, for every finite
    $F\subseteq\N$ and every $N\in\N$, there are $k\in\N\setminus F$
    and $n\ge N$ such that $(x_n)_1(k)>\varepsilon$.

    For every $n\in\N$, let $S_n\subseteq\N$ be a finite set such that $y_n$ belongs to the sublattice generated by $\{a_i:i\in S_n\}$. In particular, $(y_n)_1$ is supported on
    $S_n$. We are going to use the preceding observation inductively to
    construct strictly increasing sequences
    $(n_j)$ and $(m_j)$ such that
    \[
        (x_{n_j})_1(m_j)=(y_{n_j})_1(m_j)>\varepsilon
        \quad\text{and}\quad
        m_j\notin\bigcup_{i<j}S_{n_i}.
    \]
    The construction is straightforward. Having chosen $n_i$ and $m_i$
    for $i<j$, let
    \[
        H_{j-1}=\bigcup_{i<j}S_{n_i}\cup\{1,\ldots,m_{j-1}\},
    \]
    where $n_0=m_0=1$. Applying the preceding observation with
    $F=H_{j-1}$ and $N=n_{j-1}+1$, choose $m_j\notin H_{j-1}$ and
    $n_j>n_{j-1}$ such that
    $(x_{n_j})_1(m_j)>\varepsilon$. Since $(y_{n_j})_1$ is supported on
    $S_{n_j}$, we have $m_j\in S_{n_j}$.

    We now construct $\alpha\in PS$. This sequence will be
    supported on $\{m_j:j\in\N\}$, and its values will be chosen in
    order: at stage $j$, all coordinates in
    $\bigcup_{i<j}S_{n_i}$ have already been assigned. Let
    $\Phi_j\colon\R^{S_{n_j}}\to\R$ be the
    lattice-linear expression defining $y_{n_j}$, and let
    $w_j\in\R^{S_{n_j}}$ agree with the values of $\alpha$ already
    assigned on $S_{n_j}\cap\bigcup_{i<j}S_{n_i}$ and be zero
    elsewhere. If $e_{m_j}$ denotes the standard unit vector at
    $m_j$, then
    \[
        \Phi_j(e_{m_j})=(y_{n_j})_1(m_j)>\varepsilon.
    \]
    By positive homogeneity and continuity,
    \[
        \frac{\Phi_j(w_j+te_{m_j})}{t}
        =\Phi_j\left(\frac{w_j}{t}+e_{m_j}\right)
        \longrightarrow \Phi_j(e_{m_j})
        \quad \text{as }t\to\infty.
    \]
    If $E$ is nonempty, write
    $E=\{\beta^{(j)}:j\in\N\}$, allowing repetitions if $E$ is
    finite. Hence, by the convergence above, we may choose $t_j>0$ such that
    $\Phi_j(w_j+t_je_{m_j})>j$ and satisfying that
    $t_j\neq\beta^{(j)}(m_j)$ (if $E$ were empty we just disregard this latter requirement). Set $\alpha(m_j)=t_j$ and assign the
    value $0$ to every still unassigned coordinate in
    $S_{n_j}\setminus\{m_j\}$. After all stages, assign the value $0$
    to every remaining coordinate. Hence
    \[
    (\alpha(i))_{i\in S_{n_j}}=w_j+t_je_{m_j}.
    \]

    The $\alpha$ thus constructed certainly belongs to $PS$, and is
    such that $\alpha\notin E$, since $\alpha$ differs from
    $\beta^{(j)}$ at coordinate $m_j$, for each $j\in\N$.
    Moreover, since $(y_{n_j})_2$ only depends on the values of $\alpha$ at $S_{n_j}$, we have that
    \[
    (y_{n_j})_2(\alpha)
      =\Phi_j\bigl((\alpha(i))_{i\in S_{n_j}}\bigr)
      =\Phi_j(w_j+t_je_{m_j})>j.
    \]
    Since $\alpha\notin E$, $(x_{n_j})_2(\alpha)=(y_{n_j})_2(\alpha)
    >j$, thus showing that $((x_n)_2)$ is not bounded on $\R^{PS}$.
\end{proof}

Thus we have established the following.

\begin{thm}<\gh[theorem3_6]{https://github.com/davidmunozlahoz/Aliprantis/blob/526881fef80400db04f8f3c736dc64d3fdc00f6c/Aliprantis/Q2.lean\#L1609}>There exists a Hausdorff locally convex-solid vector lattice with
    property (B, i) whose topological completion does not have
    property (B, i).
\end{thm}

\section{Question 3}\label{sec:q3}

\subsection{Background}

Let $X$ be a Hausdorff locally solid vector lattice, and let $\hat{X}$
be its topological completion. We call an element $\hat{x}\in\hat{X}$
an \emph{upper element} if there exists an increasing sequence
$(x_n)\subseteq X_+$ such that $Jx_n$ converges to $\hat{x}$. The following result was
first proved for normed vector lattices by W.\ A.\ J.\ Luxemburg
\cite[Theorem 60.3]{luxemburg1965} and later extended by Aliprantis to
metrizable locally solid vector lattices.

\begin{thm}[\theoremcite{Lemma 4.1}{aliprantis1974}]\label{thm:lux}
    Let $X$ be a metrizable locally solid vector lattice, and let
    $\hat{X}$ be its topological completion. Then every element of
    $\hat{X}_+$ is the limit of a decreasing sequence of upper
    elements.
\end{thm}

\begin{question}[\theoremcite{Open Problem 2}{aliprantis1974}]
    Does \cref{thm:lux} remain true without assuming metrizability?
\end{question}

\subsection{Solution}

We construct a Hausdorff locally solid vector lattice whose completion
contains a positive element that is not the limit of a decreasing
sequence of upper elements.

The remaining counterexamples share a common flavor. We introduce some
notation and facts that will be useful throughout. Let $S$ be a
topological space. We denote by $C_k(S)$ the vector lattice $C(S)$ of
real-valued continuous functions on $S$, equipped with its usual
pointwise vector lattice structure and the Hausdorff locally
convex-solid topology induced by the family of seminorms
\[
p_K(f)=\sup_{x \in K}|f(x)|\qquad (f \in C_k(S)),
\]
where $K$ ranges over the compact subsets of $S$. In the literature, this is usually called the
compact-open topology.

Whenever $S$ and $T$ are topological spaces, we write $S\sqcup T$ for
their disjoint union with its standard topology. We can identify $S$ and
$T$ with their canonical homeomorphic copies in $S\sqcup T$, so that
an element of $S\sqcup T$ is either an element of $S$ or an element of
$T$. With this convention, and the analogous convention for subsets
and disjoint unions, every open subset of $S\sqcup T$ has the form
$U\sqcup V$, where $U\subseteq S$ and $V\subseteq T$ are open.
Similarly, every compact subset of $S\sqcup T$ has the form $K\sqcup
L$, where $K\subseteq S$ and $L\subseteq T$ are compact. A map
$f\colon S\sqcup T\to R$, where $R$ is a topological space, is
continuous if and only if both $f|_S$ and $f|_T$ are continuous. In
that case, we write $f=f|_S\sqcup f|_T$.

\begin{prop}\label{prop:Cdisjunion}
    Let $M$ be a metric space, let $P\subseteq M$, and set
    $Q=M\setminus P$. Consider the injective map
    \[
    \begin{array}{cccc}
    J\colon& C(M) & \longrightarrow & C_k(P\sqcup Q) \\
           & f & \longmapsto & f|_P\sqcup f|_Q \\
    \end{array}.
    \]
    Then $C(M)$, equipped with its usual pointwise vector lattice
    structure and the subspace topology induced by $J$, is a Hausdorff
    locally convex-solid vector lattice. Moreover, $C_k(P\sqcup Q)$,
    together with the embedding $J$, is the topological completion of
    $C(M)$.
\end{prop}
\begin{proof}
    Since $M$ is metrizable, so is $P\sqcup Q$; hence
    $C_k(P\sqcup Q)$ is complete
    \cite[Theorem 5.8.7]{narici_beckenstein2011}. The map $J$ is an
    injective vector lattice homomorphism. Identifying $C(M)$ with
    $J(C(M))$ and equipping it with the subspace topology makes it a
    Hausdorff locally convex-solid vector lattice, and $J$ a
    topological embedding.

    It remains to show that $J(C(M)_+)$ is dense in
    $C_k(P\sqcup Q)_+$. Let $f=f_P\sqcup f_Q\in C_k(P\sqcup Q)_+$, and
    let $K=K_P\sqcup K_Q$ be a compact subset of $P\sqcup Q$, where
    $K_P\subseteq P$ and $K_Q\subseteq Q$ are compact. The sets $K_P$
    and $K_Q$ are disjoint compact subsets of the Hausdorff space $M$,
    hence they are closed in $M$. Therefore, the function
    $h'\colon K_P\cup K_Q\to\R$ defined by $h'|_{K_P}=f_P|_{K_P}$ and
    $h'|_{K_Q}=f_Q|_{K_Q}$ is positive and continuous. Since $M$ is
    normal, the Tietze extension theorem provides a continuous
    function $h\in C(M)$ that agrees with $h'$ on $K_P\cup K_Q$. Then
    $J(|h|)|_K=f|_K$, proving the density of $J(C(M)_+)$.
\end{proof}

With the notation of the preceding proposition, let $M=\R$,
$P=\R\setminus\Q$, and let $J\colon C(M)\to C_k(P\sqcup Q)$ be the map
defined above. We show that $\chi_P\in C_k(P\sqcup Q)$ is not the
limit of a decreasing sequence of upper elements.

\begin{lem}
    Let $f \in C_k(P\sqcup Q)$ be an upper element such that $\chi
    _P\le f$. Then $f(q)\ge 1$ for every $q \in Q$.
\end{lem}
\begin{proof}
    Let $(f_n)\subseteq C(\R)_+$ be an increasing sequence converging
    to $f$ in $C_k(P\sqcup Q)$. For $t\in\R$, set
    $g(t)=\sup_n f_n(t)$. Then $g$ is lower semicontinuous on $\R$.
    Moreover, since convergence in $C_k(P\sqcup Q)$ is, in particular, pointwise, the
    restrictions of $g$ to $P$ and $Q$ coincide with those of $f$.

    Since $f\ge\chi_P$, we have $g(t)\ge1$ for all $t\in P$. Fix
    $c<1$. By lower semicontinuity, the set
    $g^{-1}(c,\infty)$ is open and contains $P$. We claim that
    $g^{-1}(c,\infty)\cap\Q$ is dense in $\Q$. Indeed, let $s\in \Q$ and $U$ an open interval containing $s$. Since every open
    interval contains an element of $P$, we have that
    $V=U\cap g^{-1}(c,\infty)\neq \emptyset$. Thus, $V$ is an open, nonempty
    subset of $\R$, so $U\cap g^{-1}(c,\infty)\cap\Q=V\cap \Q\neq\emptyset$. 
    
    Now, since
    $g|_\Q=f|_\Q$ is continuous and
    $g^{-1}(c,\infty)\cap\Q$ is dense in $\Q$, we obtain
    $f(q)=g(q)\ge c$ for every $q\in\Q$. As $c<1$ was arbitrary,
    $f(q)\ge1$ for every $q\in Q$.
\end{proof}

Fix $q_0\in Q$ and consider the open neighborhood of $0$ given by
\[
    V=\{\, f\in C_k(P\sqcup Q): |f(q_0)|<1/2\,\}.
\]
By the preceding lemma, every upper element $f\ge\chi_P$ satisfies
\[
    (f-\chi_P)(q_0)=f(q_0)\ge1.
\]
Thus $f\notin\chi_P+V$. Consequently, $\chi_P$ is not the limit of a
decreasing sequence of upper elements: every element of such a
sequence would be an upper element bounded below by $\chi_P$, whereas
no such element belongs to the neighborhood $\chi_P+V$ of $\chi_P$.
We have therefore proved the following.

\begin{thm}<\gh[theorem4_4]{https://github.com/davidmunozlahoz/Aliprantis/blob/526881fef80400db04f8f3c736dc64d3fdc00f6c/Aliprantis/Q3.lean\#L678}>There exists a Hausdorff locally convex-solid vector lattice whose
    topological completion contains a positive element that is not the
    limit of a decreasing sequence of upper elements.
\end{thm}

\section{Question 4}\label{sec:q4}

\subsection{Background}

To understand how a Hausdorff locally solid vector lattice sits in its
completion, it is useful to characterize when its image is regular.
The following result was first proved in the locally convex setting by
I.\ Kawai \cite[Theorem 4.1]{kawai1957} and later extended to the
general setting by C.\ D.\ Aliprantis
\cite[Theorem 2.3]{aliprantis1974}.

\begin{thm}[\theoremcite{Theorem 2.41}{aliprantis_burkinshaw2003}]\label{thm:kawai}
    Let $X$ be a Hausdorff locally solid vector lattice, and let $\hat{X}$
    be its topological completion with embedding $J\colon X\to
    \hat{X}$. Then the following are equivalent:
    \begin{enumerate}
        \item $J(X)$ is a regular vector sublattice of $\hat{X}$.
        \item Every topologically Cauchy net $(x_\alpha )\subseteq
            X_+$ that order converges to zero in $X$ also converges to
            zero in the topology.
    \end{enumerate}
\end{thm}

A natural weakening of condition~(ii) is the requirement that every
topologically Cauchy net $(x_\alpha )\subseteq X_+$ satisfying
$x_\alpha\downarrow0$ also satisfies $x_\alpha\to0$. C.\ D.\
Aliprantis called this the \emph{generalized (A, 0) property}; its
sequential analogue is the \emph{(A, 0) property}. He asked whether
the generalized (A, 0) property is equivalent to condition~(ii).

\begin{question}[\theoremcite{Open Problem 3}{aliprantis1974}]
    Let $X$ be a Hausdorff locally solid vector lattice satisfying the
    generalized (A, 0) property. Is its image in the topological
    completion a regular sublattice?
\end{question}

\subsection{Solution}

We construct a Hausdorff locally convex-solid vector lattice
satisfying the generalized (A, 0) property whose image in its
topological completion is not a regular sublattice.

Let $M=\{0,1\}^{\mathbb N}$ with its product topology, and let
$$
C=\{x\in M:x_{2n}=0\text{ for every }n\in\mathbb N\}.
$$
Let 
$$
P=\{x\in C:\exists N, x_n=0 \text{ for }n\geq N\}
$$ and set $Q=M\setminus P$.
Then $C$ is closed and nowhere dense in $M$, and both $P$ and $C\setminus P$ are
dense in $C$. Let $J\colon C(M)\to C_k(P\sqcup Q)$ be
the map from \cref{prop:Cdisjunion}. Equip $C(M)$ with the subspace
topology induced by $J$. Then it is a Hausdorff locally convex-solid
vector lattice whose topological completion is $C_k(P\sqcup Q)$,
together with the embedding $J$.

\begin{lem}
    The vector sublattice $J(C(M))$ is not regular in
    $C_k(P\sqcup Q)$.
\end{lem}
\begin{proof}
    Consider the set
    \[
        D=\{\, f \in C(M)_+ : f\ge \chi _{C} \, \}.
    \]
    Endow $D$ with the reverse pointwise order. Then the net
    $(f)_{f\in D}$ is decreasing in $C(M)$. We show that
    $\inf D=0$. Clearly, $0$ is a lower bound for $D$. Let
    $g\in C(M)_+$ be any lower bound for $D$, and let
    $t\in M\setminus C$. By Urysohn's lemma, there exists
    $h\in C(M)_+$ satisfying $h\ge\chi_C$ and $h(t)=0$. Since
    $g\le h$, it follows that $g(t)=0$. Thus
    $g=0$ on the dense subset $M\setminus C$, and continuity
    yields $g=0$ on $M$. Hence $\inf D=0$. On the other hand,
    $\chi_P$ is a nonzero lower bound for $J(D)$. Therefore,
    $J(C(M))$ is not regular in $C_k(P\sqcup Q)$.
\end{proof}

\begin{lem}
    $C(M)$, with the topology introduced above, has the generalized
    (A, 0) property.
\end{lem}
\begin{proof}
    Let $(f_\alpha)\subseteq C(M)_+$ be a topologically Cauchy net
    such that $f_\alpha\downarrow0$. Since $C_k(P\sqcup Q)$ is
    complete and $(J(f_\alpha))$ is also topologically Cauchy, there
    exists $g\in C_k(P\sqcup Q)$ such that $J(f_\alpha)\to g$. It
    suffices to show that $g=0$.

    First we show that $g|_Q=0$. Suppose, on the contrary, that
    $g(q_0)>0$ for some $q_0\in Q$. Then there exist
    $\varepsilon>0$ and an open neighborhood $U$ of $q_0$ in $M$
    such that $g(q)>\varepsilon$ for every $q\in U\cap Q$. Since
    $J(f_\alpha)\to g$ and $(f_\alpha)$ is decreasing, we have
    $f_\alpha(q)\ge g(q)>\varepsilon$ for every $q\in U\cap Q$ and
    every $\alpha$. As $M\setminus C$ is dense and open, there is
    a nonempty open set $V\subseteq U\setminus C$. By Urysohn's lemma,
    there exists a nonzero function $h\in C(M)_+$ such that
    $0\le h\le\varepsilon$ and $h$ vanishes outside $V$. Then
    $h\le f_\alpha$ for every $\alpha$, contradicting
    $f_\alpha\downarrow0$.

    To finish the proof, we show that $g|_P=0$. Suppose, on the
    contrary, that $g(p_0)>0$ for some $p_0\in P$. Then there exist
    $\varepsilon>0$ and an open neighborhood $U\subseteq M$ of
    $p_0$ such that $g(p)>\varepsilon$ for all $p\in U\cap P$. Hence
    $f_\alpha(p)>\varepsilon$ for every $p\in U\cap P$ and every
    $\alpha$. Fix $\alpha$. Since $f_\alpha$ is continuous on
    $M$ and $P$ is dense in $C$, we have
    $f_\alpha(t)\ge\varepsilon$ for all $t\in U\cap C$. As
    $C\setminus P$ is dense in $C$, there exists
    $s\in U\cap(C\setminus P)$. Then $s\in Q$ and
    $f_\alpha(s)\ge\varepsilon$ for every $\alpha$, which implies
    $g(s)=\inf_\alpha f_\alpha(s)\ge\varepsilon$. This contradicts the
    fact that $g|_Q=0$.
\end{proof}

In summary, we have proved the following.

\begin{thm}<\gh[theorem5_4]{https://github.com/davidmunozlahoz/Aliprantis/blob/526881fef80400db04f8f3c736dc64d3fdc00f6c/Aliprantis/Q4.lean\#L695}>There exists a Hausdorff locally convex-solid vector lattice
    satisfying the generalized (A, 0) property such that the image of
    this vector lattice in its topological completion is not regular.
\end{thm}

\section{Question 5}\label{sec:q5}

\subsection{Background}

Recall that a locally solid
vector lattice $X$ has property (A, 0) if every Cauchy sequence
$(x_n)\subseteq X_+$ satisfying $x_n\downarrow0$ also satisfies
$x_n\to0$. C.\ D.\ Aliprantis used the (A, 0) property to characterize
regularity and order density for metrizable locally solid vector
lattices in their topological completions.

\begin{thm}[\theoremcite{Theorem 6.1}{aliprantis1974}]\label{thm:regdense}
    Let $X$ be a metrizable locally solid vector lattice, let
    $\hat{X}$ be its topological completion, and let $J\colon X\to
    \hat{X}$ be the canonical embedding. Then the following are
    equivalent:
    \begin{enumerate}
        \item $X$ satisfies property (A, 0).
        \item $J(X)$ is order dense in $\hat{X}$.
        \item $J(X)$ is regular in $\hat{X}$, i.e., $J$ preserves
            arbitrary suprema.
        \item $J$ preserves countable suprema.
        \item For every $\hat{x}\in\hat{X}_+$,
            \[
                \hat{x}=\sup\{\,Jx:x\in X,\ 0\le Jx\le\hat{x}\,\}.
            \]
    \end{enumerate}
\end{thm}

It is remarkable that order density of $J(X)$ is equivalent to
regularity. Aliprantis asked whether this remains true in the
non-metrizable case.

\begin{question}[\theoremcite{Open Problem 4}{aliprantis1974}]
    Let $X$ be a Hausdorff locally solid vector lattice, let
    $\hat{X}$ be its topological completion, and let $J\colon X\to
    \hat{X}$ be the canonical embedding. Is $J(X)$ regular in
    $\hat{X}$ if and only if it is order dense in $\hat{X}$?
\end{question}

\subsection{Solution}

We construct a Hausdorff locally solid vector lattice whose image in
its completion is regular but not order dense. Since every order dense
sublattice is regular, we are going to show that the nontrivial
implication fails.

We reuse the counterexample from \cref{sec:q3}, including its
notation ($P=\R\setminus\Q$, $Q=\Q$). In fact, the following was already proved in
\cite[Example 9.8]{bilokopytov2023}, thereby solving the open question
without recognizing it as such. For completeness, we provide a
detailed proof.

\begin{prop}[\theoremcite{Example 9.8}{bilokopytov2023}]
    $J(C(\R))$ is regular in $C_k(P\sqcup Q)$, but not order dense.
\end{prop}
\begin{proof}
    We first prove that $J(C(\R))$ is regular. Let
    $F\subseteq C(\R)_+$ be such that $J(F)$ has infimum zero in
    $J(C(\R))$. We show that $J(F)$ has infimum zero in
    $C_k(P\sqcup Q)$. Suppose, on the contrary, that there exists
    $h\in C_k(P\sqcup Q)_+$ such that
    \[
        0<h\le J(f)\qquad(f\in F).
    \]
    Then $h$ is strictly positive on a nonempty open subset of either
    $P$ or $Q$. Since both $P$ and $Q$ are dense in $\R$, there are a
    nonempty open interval $V\subseteq\R$ and $\varepsilon>0$ such
    that every $f\in F$ satisfies $f\ge\varepsilon$ on $V$. Choose a
    nonzero function $\varphi\in C(\R)_+$ such that
    $\supp\varphi\subseteq V$ and $0<\varphi\le\varepsilon$. Then
    \[
        0<J(\varphi)\le J(f)\qquad(f\in F),
    \]
    contradicting the fact that $J(F)$ has infimum zero in
    $J(C(\R))$. Hence $J(C(\R))$ is regular in $C_k(P\sqcup Q)$.

    However, it is not order dense: there is no $f\in C(\R)_+$ such that
    $0<J(f)\le\chi_P$, since any such $f$ would vanish on the rational
    numbers and therefore, by continuity, be identically zero.
\end{proof}

We have thus proved the following.

\begin{thm}<\gh[theorem 6_3]{https://github.com/davidmunozlahoz/Aliprantis/blob/526881fef80400db04f8f3c736dc64d3fdc00f6c/Aliprantis/Q5.lean\#L412}>There exists a Hausdorff locally convex-solid vector lattice whose
    image in its topological completion is regular but not order
    dense.
\end{thm}

\section*{AI disclosure}

The counterexamples in this paper, except for the one in
\cref{sec:q5}, were generated by GPT-5.5, and formalized using
GPT-5.6 Sol. This paper was entirely written by the authors. GPT-5.6 Luna
and Sol were used, with human supervision, to correct typos and other
minor errors in the text.

\emergencystretch=1em
\printbibliography

@book {aliprantis_burkinshaw2003,
    AUTHOR = {Aliprantis, Charalambos D. and Burkinshaw, Owen},
     TITLE = {Locally solid {R}iesz spaces with applications to economics},
    SERIES = {Mathematical Surveys and Monographs},
    VOLUME = {105},
   EDITION = {Second},
 PUBLISHER = {American Mathematical Society, Providence, RI},
      YEAR = {2003},
     PAGES = {xii+344},
      ISBN = {0-8218-3408-8},
   MRCLASS = {46A40 (46N10 91B50)},
  MRNUMBER = {2011364},
MRREVIEWER = {Pedro\ Jim\'{e}nez Guerra},
       DOI = {10.1090/surv/105},
       URL = {https://doi.org/10.1090/surv/105},
}

@article {aliprantis1974,
    AUTHOR = {Aliprantis, Charalambos D.},
     TITLE = {On the completion of {H}ausdorff locally solid {R}iesz spaces},
   JOURNAL = {Trans. Amer. Math. Soc.},
  FJOURNAL = {Transactions of the American Mathematical Society},
    VOLUME = {196},
      YEAR = {1974},
     PAGES = {105--125},
      ISSN = {0002-9947,1088-6850},
   MRCLASS = {46A40},
  MRNUMBER = {350372},
MRREVIEWER = {D.\ H.\ Fremlin},
       DOI = {10.2307/1997016},
       URL = {https://doi.org/10.2307/1997016},
}

@article {wickstead2011,
    AUTHOR = {Wickstead, A. W.},
     TITLE = {Charalambos {D}. {A}liprantis (1946--2009)},
   JOURNAL = {Positivity},
  FJOURNAL = {Positivity. An International Mathematics Journal Devoted to
              Theory and Applications of Positivity},
    VOLUME = {15},
      YEAR = {2011},
    NUMBER = {4},
     PAGES = {539--551},
      ISSN = {1385-1292,1572-9281},
   MRCLASS = {01A70 (46-03 47-03 90-03 91-03)},
  MRNUMBER = {2861598},
       DOI = {10.1007/s11117-010-0091-7},
       URL = {https://doi.org/10.1007/s11117-010-0091-7},
}

@article {aliprantis_burkinshaw1977,
    AUTHOR = {Aliprantis, C. D. and Burkinshaw, O.},
     TITLE = {On universally complete {R}iesz spaces},
   JOURNAL = {Pacific J. Math.},
  FJOURNAL = {Pacific Journal of Mathematics},
    VOLUME = {71},
      YEAR = {1977},
    NUMBER = {1},
     PAGES = {1--12},
      ISSN = {0030-8730,1945-5844},
   MRCLASS = {46A40},
  MRNUMBER = {442633},
MRREVIEWER = {A.\ C.\ Zaanen},
       URL = {http://projecteuclid.org/euclid.pjm/1102811629},
}

@article {luxemburg_zaanen1964_X,
    AUTHOR = {Luxemburg, W. A. J. and Zaanen, A. C.},
     TITLE = {Notes on {B}anach function spaces. {X}, {XI}, {XII}, {XIII}},
      NOTE = {Nederl. Akad. Wetensch. Proc. Ser. A {\bf 67}},
   JOURNAL = {Indag. Math.},
  FJOURNAL = {},
    VOLUME = {26},
      YEAR = {1964},
     PAGES = {493--506, 507--518, 519--529, 530--543},
   MRCLASS = {46.06 (46.35)},
  MRNUMBER = {173168},
MRREVIEWER = {H.\ Gordon},
}

@book {fremlin1974,
    AUTHOR = {Fremlin, D. H.},
     TITLE = {Topological {R}iesz spaces and measure theory},
 PUBLISHER = {Cambridge University Press, London-New York},
      YEAR = {1974},
     PAGES = {xiv+266},
   MRCLASS = {46A40 (28A60 46G10)},
  MRNUMBER = {454575},
MRREVIEWER = {A.\ C.\ Zaanen},
}

@book {wnuk1999,
    AUTHOR = {Wnuk, Witold},
     TITLE = {Banach lattices with order continuous norms},
      ISBN = {83-01-12927-1},
 PUBLISHER = {Warsaw: Polish Scientific Publishers PWN},
      YEAR = {1999},
  LANGUAGE = {English},
    ZBMATH = {1327426},
       ZBL = {0948.46017},
}

@article {luxemburg1965,
    AUTHOR = {Luxemburg, W. A. J.},
     TITLE = {Notes on {B}anach function spaces. {XVI}a, {XVI}b},
      NOTE = {Nederl. Akad. Wetensch. Proc. Ser. A {\bf 68}},
   JOURNAL = {Indag. Math.},
  FJOURNAL = {},
    VOLUME = {27},
      YEAR = {1965},
     PAGES = {646--657; 658--667},
   MRCLASS = {46.06 (46.35)},
  MRNUMBER = {188770},
MRREVIEWER = {H.\ Gordon},
}

@article {kawai1957,
    AUTHOR = {Kawai, Itizo},
     TITLE = {Locally convex lattices},
   JOURNAL = {J. Math. Soc. Japan},
  FJOURNAL = {Journal of the Mathematical Society of Japan},
    VOLUME = {9},
      YEAR = {1957},
     PAGES = {281--314},
      ISSN = {0025-5645,1881-1167},
   MRCLASS = {46.00},
  MRNUMBER = {95399},
MRREVIEWER = {I.\ Namioka},
       DOI = {10.2969/jmsj/00930281},
       URL = {https://doi.org/10.2969/jmsj/00930281},
}

@book {narici_beckenstein2011,
    AUTHOR = {Narici, Lawrence and Beckenstein, Edward},
     TITLE = {Topological vector spaces},
    SERIES = {Pure and Applied Mathematics (Boca Raton)},
    VOLUME = {296},
   EDITION = {Second},
 PUBLISHER = {CRC Press, Boca Raton, FL},
      YEAR = {2011},
     PAGES = {xviii+610},
      ISBN = {978-1-58488-866-6},
   MRCLASS = {46-01 (46Axx)},
  MRNUMBER = {2723563},
MRREVIEWER = {Luis\ Manuel\ S\'{a}nchez Ruiz},
}

@book {schaefer1966,
    AUTHOR = {Schaefer, Helmut H.},
     TITLE = {Topological vector spaces},
 PUBLISHER = {The Macmillan Company, New York; Collier Macmillan Ltd.,
              London},
      YEAR = {1966},
     PAGES = {ix+294},
   MRCLASS = {46.00 (46.01)},
  MRNUMBER = {193469},
MRREVIEWER = {A.\ Pietsch},
}

@article {bilokopytov2023,
    AUTHOR = {Bilokopytov, Eugene},
     TITLE = {Locally solid convergences and order continuity of positive
              operators},
   JOURNAL = {J. Math. Anal. Appl.},
  FJOURNAL = {Journal of Mathematical Analysis and Applications},
    VOLUME = {528},
      YEAR = {2023},
    NUMBER = {1},
     PAGES = {Paper No. 127566, 23},
      ISSN = {0022-247X,1096-0813},
   MRCLASS = {46A40 (46B40 46B42 47B60)},
  MRNUMBER = {4618013},
       DOI = {10.1016/j.jmaa.2023.127566},
       URL = {https://doi.org/10.1016/j.jmaa.2023.127566},
}

@article {buskes_labuda1988,
    AUTHOR = {Buskes, G. and Labuda, I.},
     TITLE = {On {L}evi-like properties and some of their applications in
              {R}iesz space theory},
   JOURNAL = {Canad. Math. Bull.},
  FJOURNAL = {Canadian Mathematical Bulletin. Bulletin Canadien de
              Math\'{e}matiques},
    VOLUME = {31},
      YEAR = {1988},
    NUMBER = {4},
     PAGES = {477--486},
      ISSN = {0008-4395,1496-4287},
   MRCLASS = {46A40 (03E35)},
  MRNUMBER = {971576},
MRREVIEWER = {W.\ A. J. Luxemburg},
       DOI = {10.4153/CMB-1988-069-3},
       URL = {https://doi.org/10.4153/CMB-1988-069-3},
}

@article {fremlin1975,
    AUTHOR = {Fremlin, D. H.},
     TITLE = {Inextensible {R}iesz spaces},
   JOURNAL = {Math. Proc. Cambridge Philos. Soc.},
  FJOURNAL = {Mathematical Proceedings of the Cambridge Philosophical
              Society},
    VOLUME = {77},
      YEAR = {1975},
     PAGES = {71--89},
      ISSN = {0305-0041,1469-8064},
   MRCLASS = {46A40},
  MRNUMBER = {355521},
MRREVIEWER = {A.\ C.\ Zaanen},
       DOI = {10.1017/S0305004100049422},
       URL = {https://doi.org/10.1017/S0305004100049422},
}

@book {Jech,
    AUTHOR = {Jech, Thomas},
     TITLE = {Set theory},
    SERIES = {Springer Monographs in Mathematics},
   EDITION = {The third millennium edition, revised and expanded},
 PUBLISHER = {Springer-Verlag, Berlin},
      YEAR = {2003},
     PAGES = {xiv+769},
      ISBN = {3-540-44085-2},
   MRCLASS = {03Exx (03-01 03-02)},
  MRNUMBER = {1940513},
MRREVIEWER = {Eva Coplakova},
       DOI = {10.1007/3-540-44761-X},
       URL = {https://doi.org/10.1007/3-540-44761-X},
}

\end{document}